\documentclass[11pt,reqno]{amsart}
\usepackage{enumerate}
\usepackage{amsfonts}
\usepackage{amsmath}
\usepackage{amsthm}
\usepackage{amssymb}
\usepackage{latexsym}
\usepackage{multicol}
\usepackage{verbatim}
\usepackage{amscd}
\usepackage{hyperref}
\usepackage{amsmath, amscd}
\usepackage{pb-diagram}

\advance\textwidth by 1.2in \advance\oddsidemargin by -.6in \advance\evensidemargin by -.6in
\newtheorem*{cor}{Corollary}
\newtheorem*{lem}{Lemma}
\newtheorem*{prop}{Proposition}

\theoremstyle{definition}

\theoremstyle{definition}
\newtheorem*{thm}{Theorem}

\newtheorem*{rems}{Remarks}

\newcounter{cnt}
 \makeatletter
\def\mydggeometry{\makeatletter\dg@YGRID=1\dg@XGRID=20\unitlength=0.003pt\makeatother}
\makeatother \theoremstyle{remark}

\numberwithin{equation}{section}

  \DeclareMathOperator{\ad}{ad}
\let\bwdg\bigwedge
\def\bigwedge{{\textstyle\bwdg}}

\begin{document}

\newcommand{\thmref}[1]{Theorem~\ref{#1}}
\newcommand{\secref}[1]{Section~\ref{#1}}
\newcommand{\lemref}[1]{Lemma~\ref{#1}}
\newcommand{\propref}[1]{Proposition~\ref{#1}}
\newcommand{\corref}[1]{Corollary~\ref{#1}}
\newcommand{\remref}[1]{Remark~\ref{#1}}
\newcommand{\defref}[1]{Definition~\ref{#1}}
\newcommand{\er}[1]{(\ref{#1})}
\newcommand{\id}{\operatorname{id}}
\newcommand{\ord}{\operatorname{\emph{ord}}}
\newcommand{\sgn}{\operatorname{sgn}}
\newcommand{\wt}{\operatorname{wt}}
\newcommand{\tensor}{\otimes}
\newcommand{\from}{\leftarrow}
\newcommand{\nc}{\newcommand}
\newcommand{\rnc}{\renewcommand}
\newcommand{\dist}{\operatorname{dist}}
\newcommand{\qbinom}[2]{\genfrac[]{0pt}0{#1}{#2}}
\nc{\cal}{\mathcal} \nc{\goth}{\mathfrak} \rnc{\bold}{\mathbf}
\renewcommand{\frak}{\mathfrak}
\newcommand{\supp}{\operatorname{supp}}
\newcommand{\Irr}{\operatorname{Irr}}
\renewcommand{\Bbb}{\mathbb}
\nc\bomega{{\mbox{\boldmath $\omega$}}}
\nc\bpsi{{\mbox{\boldmath $\Psi$}}}
 \nc\balpha{{\mbox{\boldmath $\alpha$}}}
 \nc\bpi{{\mbox{\boldmath $\pi$}}}
\nc\bsigma{{\mbox{\boldmath $\sigma$}}}
\nc\bcN{{\mbox{\boldmath $\cal{N}$}}} \nc\bcm{{\mbox{\boldmath
$\cal{M}$}}} \nc\bLambda{{\mbox{\boldmath $\Lambda$}}}

\newcommand{\lie}[1]{\mathfrak{#1}}
\makeatletter
\def\section{\def\@secnumfont{\mdseries}\@startsection{section}{1}%
  \z@{.7\linespacing\@plus\linespacing}{.5\linespacing}%
  {\normalfont\scshape\centering}}
\def\subsection{\def\@secnumfont{\bfseries}\@startsection{subsection}{2}%
  {\parindent}{.5\linespacing\@plus.7\linespacing}{-.5em}%
  {\normalfont\bfseries}}
\makeatother
\def\subl#1{\subsection{}\label{#1}}
 \nc{\Hom}{\operatorname{Hom}}
  \nc{\mode}{\operatorname{mod}}
\nc{\End}{\operatorname{End}} \nc{\wh}[1]{\widehat{#1}}
\nc{\Ext}{\operatorname{Ext}} \nc{\ch}{\text{ch}}
\nc{\ev}{\operatorname{ev}} \nc{\Ob}{\operatorname{Ob}}
\nc{\soc}{\operatorname{soc}} \nc{\rad}{\operatorname{rad}}
\nc{\head}{\operatorname{head}}
\def\Im{\operatorname{Im}}
\def\gr{\operatorname{gr}}
\def\mult{\operatorname{mult}}
\def\Max{\operatorname{Max}}
\def\ann{\operatorname{Ann}}
\def\sym{\operatorname{sym}}
\def\Res{\operatorname{\br^\lambda_A}}
\def\und{\underline}
\def\Lietg{$A_k(\lie{g})(\bsigma,r)$}

 \nc{\Cal}{\cal} \nc{\Xp}[1]{X^+(#1)} \nc{\Xm}[1]{X^-(#1)}
\nc{\on}{\operatorname} \nc{\Z}{{\bold Z}} \nc{\J}{{\cal J}}
\nc{\C}{{\bold C}} \nc{\Q}{{\bold Q}}
\renewcommand{\P}{{\cal P}}
\nc{\N}{{\Bbb N}} \nc\boa{\bold a} \nc\bob{\bold b}
\nc\boc{\bold c} \nc\bod{\bold d} \nc\boe{\bold e}
\nc\bof{\bold f} \nc\bog{\bold g} \nc\boh{\bold h}
\nc\boi{\bold i} \nc\boj{\bold j} \nc\bok{\bold k}
\nc\bol{\bold l} \nc\bom{\bold m} \nc\bon{\bold n}
\nc\boo{\bold o} \nc\bop{\bold p} \nc\boq{\bold q}
\nc\bor{\bold r} \nc\bos{\bold s} \nc\boT{\bold t}
\nc\boF{\bold F} \nc\bou{\bold u} \nc\bov{\bold v}
\nc\bow{\bold w} \nc\boz{\bold z} \nc\boy{\bold y} \nc\ba{\bold
A} \nc\bb{\bold B} \nc\bc{\bold C} \nc\bd{\bold D} \nc\be{\bold
E} \nc\bg{\bold G} \nc\bh{\bold H} \nc\bi{\bold I} \nc\bj{\bold
J} \nc\bk{\bold K} \nc\bl{\bold L} \nc\bm{\bold M} \nc\bn{\bold
N} \nc\bo{\bold O} \nc\bp{\bold P} \nc\bq{\bold Q} \nc\br{\bold
R} \nc\bs{\bold S} \nc\bt{\bold T} \nc\bu{\bold U} \nc\bv{\bold
V} \nc\bw{\bold W} \nc\bz{\bold Z} \nc\bx{\bold x}
\nc\KR{\bold{KR}} \nc\rk{\bold{rk}} \nc\het{\text{ht }}

\nc\toa{\tilde a} \nc\tob{\tilde b} \nc\toc{\tilde c}
\nc\tod{\tilde d} \nc\toe{\tilde e} \nc\tof{\tilde f}
\nc\tog{\tilde g} \nc\toh{\tilde h} \nc\toi{\tilde i}
\nc\toj{\tilde j} \nc\tok{\tilde k} \nc\tol{\tilde l}
\nc\tom{\tilde m} \nc\ton{\tilde n} \nc\too{\tilde o}
\nc\toq{\tilde q} \nc\tor{\tilde r} \nc\tos{\tilde s}
\nc\toT{\tilde t} \nc\tou{\tilde u} \nc\tov{\tilde v}
\nc\tow{\tilde w} \nc\toz{\tilde z} \nc\bbc{\mathbb{C}}
\nc\bbn{\mathbb{N}} \nc\bbz{\mathbb{Z}} \nc\bbf{\mathbb{F}}
\nc\bbd{\mathbb{D}} \nc\f{\mathcal{F}} \nc\cs{\mathcal{S}}
\nc\cp{\mathcal{P}} \def\comp{\textnormal{comp}}
\def\span{\textnormal{span }} \nc\m{\mathcal{M}}

\title{Integral Bases for the Universal Enveloping Algebras of Map Algebras}

\author{Samuel Chamberlin}
\address{Department of Computer Science, Information Systems, and Mathematics, Park University, Parkville,
MO 64152, USA} \email{samuel.chamberlin@park.edu}
\begin{abstract}
Given a finite-dimensional, simple Lie algebra $\lie g$ over
$\bbc$ and $A$, a commutative, associative algebra with unity
over $\bbc$, we exhibit an integral form for the universal
enveloping algebra of the map algebra, $\lie g\otimes A$, and
an explicit $\bbz$-basis for this integral form. We also
produce explicit commutation formulas in the universal
enveloping algebra of $\lie{sl}_2\otimes A$ that allow us to
write certain elements in Poincar\'{e}-Birkhoff-Witt order.
\end{abstract}
\maketitle
\section{Introduction}

\noindent Let $\bbz$ denote the integers. If $A$ is an algebra,
over a field $\bbf$ of characteristic 0, define an
\underline{integral form} $A_\bbz$ of $A$ to be a
$\bbz$-algebra such that $A_\bbz\otimes_\bbz\bbf=A$. An
\underline{integral basis} for $A$ is a $\bbz$-basis for
$A_\bbz$.

The theory of integral forms for finite-dimensional simple Lie
algebras was first studied by Chevalley in 1955. His work led
to the construction of Chevalley groups (of adjoint type). The
representation theory of Chevalley groups relies on the
existence of integral forms for the universal enveloping
algebras associated to these simple finite-dimensional Lie
algebras. In 1966, suitable integral forms were discovered by
Cartier and Kostant independently. They obtained precise
information about these integral forms through integral bases
(bases whose $\bbz$-span is the integral form). The
construction of such bases relies heavily on straightening
identities in the universal enveloping algebra, which allow one
to write certain elements in Poincar\'{e}-Birkhoff-Witt (PBW)
order. Cartier and Kostant's $\bbz$-form led to the
construction of Lie groups and Lie algebras over a field of
positive characteristic, generalizing Chevalley's construction.
This in turn led to the development of representation theory
over a field of positive characteristic, \cite{H}.

Also in 1966, Serre showed that a finite dimensional Lie
algebra can be presented by generators and relations determined
solely by the Cartan matrix. With a generalized Cartan matrix
one can use the Serre presentation to define the class of
Kac-Moody Lie algebras. The most widely studied subclass of
these algebras, the simple affine Lie algebras, have structure
and representation theories similar to those for simple finite
dimensional Lie algebras. The best way to understand untwisted
simple affine Lie algebras is as central extensions of loop
algebras, \cite{C}. For these affine Lie algebras Garland, in
1978, extended the theory of integral forms started by
Chevalley and continued by Cartier and Kostant. Garland also
gave explicit constructions of $\bbz$-bases for these integral
forms via a Chevalley-type basis for the affine Lie algebra.
The complexity in formulating integral bases and straightening
identities increased greatly in the affine case. These results
were then extended to all simple affine Lie algebras by
Mitzman, in 1983. In 2007, Jakeli\'{c} and Moura used Garland
and Mitzman's work on integral forms to study representations
of affine Lie algebras over a field of positive characteristic,
\cite{JM}.

Recently there has been much interest in map algebras and their
representations, \cite{CFK}, \cite{NSS}. A map algebra is a Lie
algebra $\lie g\otimes_\bbc A$, where $\lie g$ is any
finite-dimensional simple complex Lie algebra and $A$ is any
commutative associative complex algebra. The Lie bracket is
given by
$$[z\otimes a, z'\otimes b]=[z,z']\otimes ab,\ z,z'\in\lie a,\ a,b\in A$$
These Lie algebras are so named because if $X$ is an algebraic
variety and $A$ is its coordinate ring then $\lie g\otimes A$
can also be realized as the Lie algebra of regular maps
$X\to\lie g$ with pointwise Lie bracket.

Map algebras are a generalization of the loop algebras, for
which $A$ is the Laurent polynomials. Therefore it is natural
to wish to generalize the Garland's work to the map algebras.
We formulate and prove straightening identities in the
universal enveloping algebra of $\lie{sl}_2\otimes A$. The
notational difficulties increase greatly when one moves to the
general case. Additionally the formula we have proved is much
more general than the one proved by Garland in \cite{G}.

If $A$ has a $\bbc$-basis which is closed under multiplication,
these straightening identities lead to the construction of an
integral form and integral basis for the universal enveloping
algebra of $\lie g\otimes A$.

A natural offshoot of this work will be to study
representations for map algebras over a field of positive
characteristic.

\section{Preliminaries}

The following notation will be used throughout this manuscript:
$\bbc$ is the set of complex numbers and $\bbz_{\geq0}$ is the
set of non-negative integers. Given any Lie algebra $\lie a$,
$\bu(\lie a)$ is the universal enveloping algebra of $\lie a$.

Let $\lie g$ be a finite-dimensional, complex simple Lie
algebra of rank $n$ where $I=\{1,\ldots,n\}$. Fix a Cartan
subalgebra $\lie h$ of $\lie g$ and let $R$ denote the
corresponding set of roots. Let $\{\alpha_i\}_{i\in I}$ be a
set of simple roots and $Q$ (respectively $Q^+)$, be the
integer span (respectively $\bbz_{\geq0}$-span) of the simple
roots. Set $R^+=R\cap Q^+$ and fix an order on
$R^+=\{\beta_1,\ldots,\beta_m\}$.

Let $\{x_\alpha^{\pm},h_i:\alpha\in R^+,\ i\in I\}$ be a
Chevalley basis of $\lie g$ and set
$x_i^\pm=x_{\alpha_i}^{\pm}$, and
$h_\alpha=[x_\alpha^+,x_\alpha^-]$. Note that
$h_i=h_{\alpha_i}$. For each $\alpha\in R^+$, the subalgebra of
$\lie g$ spanned by $\{x_{\alpha}^\pm,h_\alpha\}$ is isomorphic
to $\lie{sl}_2$ (When $\lie g=\lie{sl}_2$ we write the
$\bbc$-basis as $\{x^-,h,x^+\}$). Set
$$\lie n^\pm=\bigoplus_{\alpha\in R^+}\bbc x_\alpha^\pm,$$
and note that $\lie g=\lie n^-\oplus \lie h\oplus \lie n^+.$

By the Poincare Birkhoff Witt theorem, we know that if $\lie b$
and $\lie c$ are Lie subalgebras of $\lie a$ such that $\lie
a=\lie b\oplus\lie c$ as vector spaces then
$$\bu(\lie a)\cong\bu(\lie b)\otimes\bu(\lie c)$$ as vector
spaces. So in particular we have a vector space isomorphism
$$\bu(\lie g)\cong\bu(\lie n^-)\otimes\bu(\lie
h)\otimes\bu(\lie n^+)$$

Given any $u\in\bu(\lie a)$ and $r\in\bbz_{\geq0}$ define
$u^{(r)}=\frac{u^r}{r!}$ and
$$\binom{u}{r}=\frac{u(u-1)\cdots(u-r+1)}{r!}.$$

Define $T^0(\lie a)=\bbc$, and for all $j\geq1$, define
$T^j(\lie a)=\lie a^{\otimes j}$, $T(\lie
a)=\bigoplus_{j=0}^\infty T^j(\lie a)$, and $T_j(\lie
a)=\bigoplus_{k=0}^jT^k(\lie a)$. Then set $\bu_j(\lie a)$ to
be the image of $T_j(\lie a))$ under the canonical surjection
$T(\lie a)\to\bu(\lie a)$. Then for any $u\in\bu(\lie a)$
\emph{define the degree of $u$} by
$$\deg u=\min_{j}\{u\in\bu_j(\lie a)\}$$

\subsection{Map Algebras}

Fix $A$ a commutative associative algebra with unity over
$\bbc$. Let $\lie a$ be a Lie algebra, over $\bbc$, with Lie
bracket $[\ ,\ ]_{\lie a}$. The \emph{map algebra} of $\lie a$
is the vector space $\lie a\otimes A$, with Lie bracket given
by
$$[z\otimes a, z'\otimes b]=[z,z']_{\lie a}\otimes ab,\ z,z'\in\lie a,\ a,b\in A.$$
$\lie a$ can be embedded in this Lie algebra as $\lie a\otimes
1$. In the case that $A$ is the coordinate ring of some
algebraic variety $X$ $\lie{a}\otimes A$ is isomorphic to the
Lie algebra of regular maps $X\to\lie a$.

Note that by the PBW Theorem we have a vector space isomorphism
$$\bu(\lie g\otimes A)\cong\bu(\lie n^+\otimes
A)\otimes\bu(\lie h\otimes A)\otimes\bu(\lie n^-\otimes A)$$

For each $\alpha\in R^+$, let
$\Omega_\alpha:\bu(\lie{sl}_2\otimes A)\to\bu(\lie g\otimes A)$
be the algebra homomorphism defined by
$$x^\pm\otimes a\mapsto x^\pm_\alpha\otimes a\hskip.3in h\otimes a\mapsto h_\alpha\otimes a$$

\subsection{Multisets}

Given any set $S$ define a \emph{multiset of elements of $S$}
to be a multiplicity function $\chi:S\to\bbz_{\geq0}$. Define
$\f(S)=\{\chi:S\to\bbz_{\geq0}|\supp\chi<\infty\}$. For
$\chi\in\f(S)$ define $|\chi|=\sum_{s\in S}\chi(s)$. Notice
that $\f(S)$ is an abelian monoid under function addition.
Define a partial order on $\f(S)$ so that for
$\psi,\chi\in\f(S)$, $\psi\leq\chi$ if $\psi(s)\leq\chi(s)$ for
all $s\in S$. Define
$$\f=\{\chi:A\to\bbz_{\geq0}|\supp\chi<\infty\},\hskip.3in
\f_k=\{\chi\in\f:|\chi|=k\}$$ and given $\chi\in\f$ define
$$\f(\chi)=\{\psi\in\f:\psi\leq\chi\},\hskip.3in
\f_k(\chi)=\{\psi\in\f(\chi):|\psi|=k\}$$ If $\psi\in\f(\chi)$
we define $\chi-\psi$ by standard function subtraction. Also
define functions $\pi:\f\to A$, $\m:\f\to\bbz_{>0}$ and, for
all $\alpha\in R^+$, $x^\pm_\alpha:\f\to\bu(\lie{g}\otimes A)$
by
$$\pi(\psi)=\prod_{a\in A}a^{\psi(a)}\hskip.5in \m(\psi)=\frac{|\psi|!}{\prod\psi(a)!}
\hskip.5in x^\pm_\alpha(\psi)=\prod_{a\in A}(x_\alpha^\pm\otimes a)^{(\psi(a))}$$

\section{An Integral Form and Integral Basis}

In section we will define our integral form and give our
integral basis for the map algebra, $\lie g\otimes A$.

\subsection{Definition of $p(\varphi,\chi)$}

Given $\varphi,\chi\in\f$, Recursively define functions
$p:\f^2\to\bu(h\otimes A)\subset\bu(\lie{sl}_2\otimes A)$ by
$p(0,0)=1$ and for $\varphi,\chi\in\f-\{0\}$,
\begin{eqnarray*}
p(\varphi,\chi)&=&-\frac{\delta_{|\varphi|,|\chi|}}{|\varphi|}
\sum_{\substack{\psi_1\in\f(\varphi)-\{0\}\\ \psi_2\in\f(\chi)-\{0\}}}
\m(\psi_1)\m(\psi_2)\left(h\otimes \pi(\psi_1)\pi(\psi_2)\right)p(\varphi-\psi_1,\chi-\psi_2)
\end{eqnarray*}

Define $p(\chi)=p(\chi,|\chi|\chi_1)$ and, for all $\alpha\in
R^+$, $p_\alpha(\chi)=\Omega_{\alpha}(p(\chi))\in\bu(\lie
h\otimes A)\subset\bu(\lie g\otimes A)$.

\begin{rems}
$(a)$ $p(\varphi,\chi)=0$ if $|\varphi|\neq|\chi|$.

$(b)$ The $p(\varphi,\chi)$ are a generalization of Garland's
$\Lambda_k(H_\alpha(r))$ because
$p\left(k\chi_{t^r}\right)=\Lambda_k(H_\alpha(r))$, where
$\chi_a$ is the characteristic function on $a\in A$, \cite{G}.
\end{rems}

\subsection{An Integral Form and Integral Basis}

If $A$ has a $\bbc$-basis, $\bb$, which is closed under
multiplication define $\bu_\bbz(\lie g\otimes A)$, to be the
$\bbz$-subalgebra of $\bu(\lie g\otimes A)$ generated by
$$\left\{\left(x_\alpha^\pm\otimes b\right)^{(r)}:\alpha\in R^+,\ b\in\bb,\ r\in\bbz_{\geq0}\right\}$$

Fix an order on $R^+=\{\beta_1,\ldots,\beta_m\}$. Then define
functions $f^\pm:\f(\bb)^{\times m}\to\bu(\lie g\otimes A)$ by
$$(\psi_1,\ldots,\psi_m)\mapsto x_{\beta_1}^\pm(\psi_1)x_{\beta_2}^\pm(\psi_2)\cdots x_{\beta_m}^\pm(\psi_m)$$
and $f^0:\f(\bb)^{\times n}\to\bu(\lie h\otimes A)$ by
$$(\psi_1,\ldots,\psi_n)\mapsto p_1(\psi_1)p_2(\psi_2)\cdots p_n(\psi_n)$$

Define
$\mathcal{B}=\left\{f^-(\und{\psi})f^0(\und{\chi})f^+(\und{\psi}')|\und{\psi},\und{\psi'}\in\f(\bb)^{\times
m},\ \und{\chi}\in\f(\bb)^{\times n}\right\}$ Define
$\mathcal{B}_\pm$ to be the set consisting of all
$f^\pm(\und{\psi})$, and $\mathcal{B}_0$ to be the set
consisting of all $f^0(\und{\chi})$.

\begin{thm}\label{thm}
$\mathcal{B}$ is a $\bbz$-basis for $\bu_\bbz(\lie g\otimes
A)$.
\end{thm}

Proposition \ref{degq} and the Poincar\'{e}-Birkhoff-Witt
theorem easily give the following corollary to this theorem.
\begin{cor}
$\bu_\bbz(\lie g\otimes A)$ is an integral form for $\bu(\lie
g\otimes A)$.
\end{cor}

The remainder of this chapter is devoted to the proof of
Theorem \ref{thm}. First we show that
$\mathcal{B}\subset\bu_\bbz(\lie g\otimes A)$. This allows us
to define $\bbz$-subalgebras $\bu_\bbz^+(\lie g\otimes A)$,
$\bu_\bbz^-(\lie g\otimes A)$, and $\bu_\bbz^0(\lie g\otimes
A)$. Then we prove that $\mathcal{B}_+$, $\mathcal{B}_-$ and
$\mathcal{B}_0$ are $\bbz$-bases of these subalgebras
respectively. Finally, we prove a triangular decomposition
$\bu_\bbz(\lie g\otimes A)=\bu_\bbz^-(\lie g\otimes
A)\bu_\bbz^0(\lie g\otimes A)\bu_\bbz^+(\lie g\otimes A)$. This
means that as $\bbz$-modules $\bu_\bbz(\lie g\otimes
A)\cong\bu_\bbz^-(\lie g\otimes A)\otimes\bu_\bbz^0(\lie
g\otimes A)\otimes\bu_\bbz^+(\lie g\otimes A)$. The theorem
follows.

\section{The Subalgebras $\bu_\bbz^\pm(\lie g\otimes A)$}

Define $\bu_\bbz^\pm(\lie g\otimes A)$ to be the
$\bbz$-subalgebra of $\bu_\bbz(\lie g\otimes A)$ generated by
$$\left\{\left(x_\alpha^\pm\otimes b\right)^{(r)}:\alpha\in
R^+,\ b\in\bb,\ r\in\bbz_{\geq0}\right\}.$$

\subsection{$\mathcal{B}_\pm$ is a $\bbz$-Basis for $\bu_\bbz^\pm(\lie g\otimes A)$}

\begin{lem}\label{pmbasislem}
Let $a,b\in\bb$, $r,s\in\bbz_{\geq0}$, and $\alpha,\beta\in
R^+$ be
given. Define\\
$R_{\alpha,\beta}=\{i\alpha+j\beta:i,j\in\bbz\}\cap R$. Then
the following identities hold:

$(a)$ If $R_{\alpha,\beta}$ is of type $A_2$
\begin{eqnarray*}
x_\alpha^\pm(r\chi_a)x_\beta^\pm(s\chi_b)
&=&\sum_{k=0}^{\min(r,s)}\varepsilon_k\left(x_\beta^\pm\otimes b\right)^{(s-k)}
\left(x_{\alpha+\beta}^\pm\otimes ab\right)^{(k)}\left(x_\alpha^\pm\otimes a\right)^{(r-k)}
\end{eqnarray*}
where $\varepsilon_k\in\{1,-1\}$, for all $k$.

$(b)$ If $R_{\alpha,\beta}$ is of type $B_2$, then
\begin{eqnarray*}
x_\alpha^\pm(r\chi_a)x_\beta^\pm(s\chi_b)
&=&\sum\varepsilon_{k_1,k_2}\left(x_\beta^\pm\otimes b\right)^{(s-k_1-k_2)}
\prod_{j=1}^2\left(x_{j\alpha+\beta}^\pm\otimes a^jb\right)^{(k_j)}\left(x_\alpha^\pm\otimes a\right)^{(r-k_1-2k_2)}
\end{eqnarray*}
where the sum is over all $k_1,k_2\in\bbz_{\geq0}$ such that
$k_1+k_2\leq s$ and $k_1+2k_2\leq r$, and
$\varepsilon_{k_1,k_2}\in\{1,-1\}$, for all $k_1,k_2$.

$(c)$ If $R_{\alpha,\beta}$ is of type $G_2$, then
\begin{eqnarray*}
x_\alpha^\pm(r\chi_a)x_\beta^\pm(s\chi_b)
&=&\sum\varepsilon_{k_1,k_2,k_3,k_4}\left(x_\beta^\pm\otimes b\right)^{\left(s-\sum_{j=1}^3k_j-2k_4\right)}
\prod_{j=1}^3\left(x_{j\alpha+\beta}^\pm\otimes a^jb\right)^{(k_j)}\\
&\times&\left(x_{3\alpha+2\beta}^\pm\otimes a^3b^2\right)^{(k_4)}
\left(x_\alpha^\pm\otimes a\right)^{\left(r-\sum_{j=1}^3jk_j-3k_4\right)}
\end{eqnarray*}
where the sum is over all $k_1,k_2,k_3,k_4\in\bbz_{\geq0}$ such
that $k_1+k_2+k_3+2k_4\leq s$ and  $k_1+2k_2+3k_3+3k_4\leq r$,
and $\varepsilon_{k_1,k_2,k_3,k_4}\in\{1,-1\}$, for all
$k_1,k_2,k_3,k_4$.
\end{lem}

The proof of this lemma proceeds by induction on $r$ for the
case $s=1$ and then by induction on $s$ for a general $r$. The
details can be found in \cite{Ch}.

\begin{cor}\label{pmbasis}
$\mathcal{B}_\pm$ is a $\bbz$-basis for $\bu_\bbz^\pm(\lie
g\otimes A)$.
\end{cor}
\begin{proof}
It will suffice to show that any product of elements of the set
$\{(x_\alpha^\pm\otimes b): \alpha\in R^+,\ b\in\bb\}$ is in
the $\bbz$-span of $\mathcal{B}_\pm$. Clearly, for all
$\alpha\in R^+$ and $b\in\bb$,
$$(x_\alpha^\pm\otimes b)^{(r)}(x_\alpha^\pm\otimes b)^{(s)}=\binom{r+s}{r}(x_\alpha^\pm\otimes
b)^{(r+s)}$$ So it will suffice to show if $\alpha,\beta\in
R^+$, $a,b\in\bb$ and $r,s\in\bbz_{\geq0}$ then
$\left[(x_\alpha^\pm\otimes a)^{(r)},(x_\beta^\pm\otimes
b)^{(s)}\right]$ is in the $\bbz$-span of $\mathcal{B}_\pm$ and
has degree less than $r+s$. If $\alpha+\beta\notin R^+$ this
claim is trivially true. If $\alpha+\beta\in R^+$ this claim
can be shown by induction on $r+s$. If $r+s\leq1$ the claim is
trivially true. The inductive step is true by the lemma.
\end{proof}

\section{A Straightening Lemma}

Define functions $D^\pm:\f^3\to\bu(\lie{sl}_2\otimes A)$ by
\begin{eqnarray*}
D^\pm(\psi_1,\psi_2,0)&=&\delta_{|\psi_1|+|\psi_2|,0}\\
D^\pm(\psi_1,\psi_2,\chi_b)&=&\delta_{|\psi_1|,|\psi_2|}\m(\psi_1)\m(\psi_2)
\left(x^\pm\otimes b\pi(\psi_1)\pi(\psi_2)\right)\\
D^\pm(\psi_1,\psi_2,\psi_3)&=&\frac{1}{|\psi_3|}\sum_{\substack{\phi_1\in\f(\psi_1)\\ \phi_2\in\f(\psi_2)}}
\sum_{b\in\supp\psi_3}D^\pm(\phi_1,\phi_2,\chi_b)D^\pm\left(\psi_1-\phi_1,\psi_2-\phi_2,\psi_3-\chi_b\right)
\end{eqnarray*}

\begin{rems}
(a) $D^\pm(\psi_1,\psi_2,\psi_3)=0$ if
$|\psi_1|\neq|\psi_2|$.\\
(b) The $D^\pm(\psi_1,\psi_2,\psi_3)$ are a generalization of
Garland's $D_k^\pm$ as
$D^+(k\chi_t,k\chi_1,r\chi_1)=D^+_k\left(\xi^{(r)}\right)$ and
$D^-(k\chi_t,k\chi_1,r\chi_t)=D^-_k\left(\xi^{(r)}\right)$,
\cite{G}.
\end{rems}

\subsection{An Explicit Formula for
$D^\pm(\psi,|\psi|\chi_b,k\chi_c)$}

\textbf{Definition.} Given $\chi\in\f$ define
$$\mathcal{P}(\chi)=\left\{\psi\in\f(\f):\sum_{\phi\in\f}\psi(\phi)\phi=\chi\right\}$$
and $\cp_k(\chi)=\cp(\chi)\cap\f_k(\f)$. A \emph{partition of
$\chi$} is an element of $\cp(\chi)$.

The following proposition can easily be shown by induction on
$k$.

\begin{prop}\label{Dpmexp}
For all $\psi\in\f$, $k\in\bbz_{\geq0}$ with $k\geq1$, and
$b,c\in A$.
$$D^\pm(\psi,|\psi|\chi_b,k\chi_c)=\sum_{\chi\in\cp_k(\psi)}
\prod_{\phi\in\f}\left(\m(\phi)\left(x^\pm\otimes cb^{|\phi|}\pi(\phi)\right)\right)^{(\chi(\phi))}$$
\end{prop}

\subsection{The Degree of $D^\pm(\psi_1,\psi_2,\psi_3)$}

The following proposition can easily be shown by induction on
$|\psi_3|$.

\begin{prop}\label{degD}
If $\alpha\in R^+$, and $\psi_1,\psi_2,\psi_3\in\f$ then
$D^\pm(\psi_1,\psi_2,\psi_3)$ is homogeneous of degree
$|\psi_3|$.
\end{prop}

\subsection{The Definition of $\bbd(\psi_1,\psi_2,\psi_3)$}

Define $\bbd:\f^3\to\bu(\lie{sl}_2\otimes A)$ by
\begin{eqnarray*}
\bbd(\psi_1,\psi_2,\psi_3)&=&\sum_{\substack{\phi_1\in\f(\psi_1)\\ \phi_2\in\f(\psi_2)}}p(\phi_1,\phi_2)
D^+(\psi_1-\phi_1,\psi_2-\phi_2,\psi_3)
\end{eqnarray*}

Using the previous two propositions it can be easily shown
that, for all $\psi_1,\psi_2,\psi_3\in\f$,
$\bbd(\psi_1,\psi_2,\psi_3)$ has degree less than or equal to
$|\psi_3|+|\psi_1|$

Set
$$D_\alpha^\pm\left(\psi_1,\psi_2,\psi_3\right)=\Omega_\alpha\left(D^\pm\left(\psi_1,\psi_2,\psi_3\right)\right)$$
$$\bbd_\alpha\left(\psi_1,\psi_2,\psi_3\right)=\Omega_\alpha\left(\bbd\left(\psi_1,\psi_2,\psi_3\right)\right)$$

\subsection{The Straightening Lemma}

\begin{lem}\label{basic}
For all $\varphi,\chi\in\f$.
$$x^+(\varphi)x^-(\chi)=
\sum_{\substack{\psi_1,\psi_2,\phi_1,\phi_2\in\f\\ \psi_1+\psi_2\in\f(\chi)\\ \phi_1+\phi_2\in\f(\varphi)}}
(-1)^{|\psi_1|+|\psi_2|}D^-(\phi_1,\psi_1,\chi-\psi_1-\psi_2)\bbd(\phi_2,\psi_2,\varphi-\phi_1-\phi_2)$$
\end{lem}

This lemma is proved in a later section. Before proving the
lemma we will state and prove some corollaries and other
necessary lemmas.

\begin{cor}\label{qinuz}
For all $\alpha\in R^+$, and $\varphi,\chi,\psi\in\f$\\
\indent$(i)_{\chi}$
$$D^+_\alpha\left(\varphi,\chi,\psi\right),
D^-_\alpha\left(\varphi,\chi,\psi\right)\in\bu_\bbz(\lie g\otimes A)$$

$(ii)_{\chi}$
$$p_\alpha\left(\varphi,\chi\right)\in\bu_\bbz(\lie g\otimes A)$$
\end{cor}
\noindent\textit{Proof of corollary.} We will prove both
statements simultaneously by induction on $k$ according to the
scheme
$$(ii)^{k}\Rightarrow(i)^{k}\Rightarrow(ii)^{k+1},$$
where $(i)^{k}$ is the statement that $(i)_{\chi}$ holds for
all $\chi\in\f$ with $|\chi|\leq k$, and similarly for
$(ii)^{k}$. $(ii)^1$ is trivially true. For $(i)_0$ we have, by
induction on $|\psi|$,
$D^\pm(0,0,\psi)=x^\pm(\psi)\in\bu_\bbz(\lie g\otimes A)$.
Assume that $(i)^{k-1}$ and $(ii)^{k}$ hold for some $k\geq1$.
Note that $(i)_\chi$ is trivially true if $|\varphi|\neq|\chi|$
so assume that they have the same size. Then, by Lemma
\ref{basic}, we have for all $\varphi,\chi\in\f(k)$ and all
$\psi\in\f$
\begin{eqnarray*}
&&\hskip-.6inx^+(\varphi)x^-(\chi+\psi)\\
&=&\sum_{\substack{\psi_1,\psi_2,\phi_1,\phi_2\in\f\\ \psi_1+\psi_2\in\f(\chi+\psi)\\
\phi_1+\phi_2<\varphi\\ \psi_1+\psi_2\neq\chi}}
(-1)^{|\psi_1|+|\psi_2|}D^-(\phi_1,\psi_1,\chi+\psi-\psi_1-\psi_2)\bbd(\phi_2,\psi_2,\varphi-\phi_1-\phi_2)\\
&+&\sum_{\substack{\psi_1,\psi_2,\phi_1,\phi_2\in\f\\ \psi_1+\psi_2=\chi\\ \phi_1+\phi_2=\varphi\\ \phi_1<\varphi}}
(-1)^{|\chi|}D^-(\phi_1,\psi_1,\psi)p(\phi_2,\psi_2)+(-1)^{|\chi|}D^-(\varphi,\chi,\psi)
\end{eqnarray*}
The left side is clearly in $\bu_\bbz(\lie{sl}_2\otimes A)$.
The end sums are in $\bu_\bbz(\lie{sl}_2\otimes A)$ by
$(i)^{k-1}$ and $(ii)^k$. Thus
$D^-(\varphi,\chi,\psi)\in\bu_\bbz(\lie{sl}_2\otimes A)$. Also
by Lemma \ref{basic}

\begin{eqnarray*}
&&\hskip-.6in x^+(\varphi+\psi)x^-(\chi)\\
&=&\sum_{\substack{\psi_1,\psi_2,\phi_1,\phi_2\in\f\\ \psi_1+\psi_2<\chi\\
\phi_1+\phi_2\in\f(\varphi+\psi)\\ \phi_1+\phi_2\neq\varphi}}
(-1)^{|\psi_1|+|\psi_2|}D^-(\phi_1,\psi_1,\chi-\psi_1-\psi_2)\bbd(\phi_2,\psi_2,\varphi+\psi-\phi_1-\phi_2)\\
&+&\sum_{\substack{\phi\in\f(\chi)-\{0\}\\ \phi'\in\f(\varphi)-\{0\}}}
(-1)^{|\varphi|}p(\phi,\phi')D^+(\chi-\phi,\varphi-\phi',\psi)+(-1)^{|\varphi|}D^+(\chi,\varphi,\psi)
\end{eqnarray*}

Again the left side is clearly in $\bu_\bbz(\lie{sl}_2\otimes
A)$. The end sums are in $\bu_\bbz(\lie{sl}_2\otimes A)$ by
$(i)^{k-1}$ and $(ii)^k$. Thus
$D^\pm(\varphi,\chi,\psi)\in\bu_\bbz(\lie{sl}_2\otimes A)$.
Applying $\Omega_\alpha$ we get $(i)^k$. Thus $(ii)^k$ implies
$(i)^k$. Assume that $(i)^{k}$ and $(ii)^k$ hold and let
$\varphi,\chi\in\f(k+1)$ be given. Then, again by the previous
lemma
\begin{eqnarray*}
x^+(\varphi)x^-(\chi)
&=&\sum_{\substack{\psi_1,\psi_2,\phi_1,\phi_2\in\f\\ \psi_1+\psi_2<\chi\\ \phi_1+\phi_2<\varphi}}
(-1)^{|\psi_1|+|\psi_2|}D^-(\phi_1,\psi_1,\chi-\psi_1-\psi_2)\bbd(\phi_2,\psi_2,\varphi-\phi_1-\phi_2)\\
&+&(-1)^{|\chi|}p(\varphi,\chi)
\end{eqnarray*}
Again the left side is clearly in $\bu_\bbz(\lie{sl}_2\otimes
A)$. The end sum is in $\bu_\bbz(\lie{sl}_2\otimes A)$ by
$(i)^{k}$ and $(ii)^k$. Thus
$p(\varphi,\chi)\in\bu_\bbz(\lie{sl}_2\otimes A)$. Applying
$\Omega_\alpha$ we get $(ii)^{k+1}$. Thus $(i)^k$ implies
$(ii)^{k+1}$.\hskip.5in$\square$

\section{$\mathcal{B}_0$ is a $\bbz$-basis for $\bu_\bbz^0(\lie g\otimes A)$}

Define $\bu_\bbz^0(\lie g\otimes A)$ to be the
$\bbz$-subalgebra of $\bu_\bbz(\lie g\otimes A)$ generated by
the set $\{p_i(\chi):\chi\in\f(\bb),\ i\in I\}$.

For the remainder of this section $\lie g=\lie{sl}_2$ unless
stated otherwise.

We adapt Garland's proof of the corresponding fact in \cite{G},
section 9.

\subsection{$\bu(h\otimes A)=\bbc-\span\mathcal{B}_0$}

\begin{prop}\label{degq}
For all $\chi\in\f$
$$p(\chi)=(-1)^{|\chi|}\prod_{a\in A}(h\otimes a)^{(\chi(a))}+\textnormal{ elments of }\bu(h\otimes A)
\textnormal{ of degree less than }|\chi|$$
\end{prop}

\begin{lem}
$$\bu(h\otimes A)=\bbc-\span\mathcal{B}_0$$
\end{lem}
\begin{proof}
Clearly $\bu(h\otimes A)\subset\bbc-\span\mathcal{B}_0$. For
the other inclusion it suffices to show that
$$\prod_{a\in A}(h\otimes
a)^{(\chi(a))}\in\bbc-\span\mathcal{B}_0$$ for all $\chi\in\f$.
This can be shown by Proposition \ref{degq} and induction on
$|\chi|$.
\end{proof}

\subsection{The Adjoint Action on the $\bbz$-Span of the Chevalley Basis}

For this subsection let $\lie g$ be arbitrary. Define $(\lie
g\otimes A)_\bbz$ to be the $\bbz$-span of the set
$S=\{(x^\pm_{\alpha}\otimes b),\ (h_i\otimes b):b\in\bb,\ i\in
I,\ \alpha\in R^+\}$.

\begin{lem}
For all $\alpha\in R^+$, $r\in\bbz_{\geq0}$ and $b\in\bb$,
$$\ad\left(x^\pm_{\alpha}\otimes b\right)^{(r)}((\lie g\otimes
A)_\bbz)\subset(\lie g\otimes A)_\bbz$$
\end{lem}

\begin{cor}
For all $C\in\bu_\bbz(\lie g\otimes A)$, $\ad C((\lie g\otimes
A)_\bbz)\subset(\lie g\otimes A)_\bbz$.
\end{cor}

\subsection{The Adjoint Action on $((\lie g\otimes A)_\bbz)^{\otimes r}$}

We note the following simple lemma.

\begin{lem}
Let $V,W$ be $\lie g\otimes A$-modules, with respective
additive subgroups $M,N$. If $M,N$ are preserved by
$\left(x^\pm_{\alpha}\otimes b\right)^{(r)}$, for all
$\alpha\in R^+$, $b\in\bb$ and $r\in\bbz_{\geq0}$, then
$M\otimes N\subset V\otimes W$ is also preserved by these
elements.
\end{lem}

Let $\ad^{(r)}$ denote the action of $\lie g\otimes A$ (and of
$\bu(\lie g\otimes A)$) induced by adjoint action, on $(\lie
g\otimes A)^{\otimes r}$. Then we have the following corollary
to the previous two lemmas.

\begin{cor}\label{adQ}
For all $Q\in\bu_\bbz(\lie g\otimes A)$ and all
$r\in\bbz_{\geq0}$,
$$\ad^{(r)}Q\left(\left((\lie g\otimes A)_\bbz\right)^{\otimes r}\right)
\subset\left((\lie g\otimes A)_\bbz\right)^{\otimes r}$$
\end{cor}

\subsection{$p(\chi)p(\chi')\in\bbz-\span\mathcal{B}_0$}

\begin{lem}
For all $\chi,\chi'\in\f(\bb)$,
$p(\chi)p(\chi')\in\bbz-\span\mathcal{B}_0$
\end{lem}
The proof is similar to that found in \cite{G} Lemma 9.2.

The following corollary follows directly
\begin{cor}
$$p(\chi)p(\chi')-\prod_{a\in
A}\binom{(\chi+\chi')(a)}{\chi(a)}p(\chi+\chi')\in\bbz-\span\mathcal{B}_0$$
\end{cor}

\subsection{$\mathcal{B}_0$ is a $\bbz$-basis for $\bu^0_\bbz(\lie g\otimes A)$}

\begin{lem}
$\mathcal{B}_0$ is a $\bbz$-basis for $\bu^0_\bbz(\lie g\otimes
A)$.
\end{lem}
\begin{proof}
It will suffice to show that any product of elements of
$\{p_i(\chi):\chi\in\f(\bb)\}$ is in the $\bbz$-span of
$\bu^0_\bbz(\lie g\otimes A)$. To show this we will proceed by
induction on the degree of such a product. Since $p_i(\chi)$
and $p_j(\chi')$ commute it will suffice to apply
$\Omega_{\alpha_i}$ to the previous corollary because by
Proposition \ref{degq} $$p(\chi)p(\chi')-\prod_{a\in
A}\binom{(\chi+\chi')(a)}{\chi(a)}p(\chi+\chi')$$ has degree
less than $|\chi|+|\chi'|$.
\end{proof}

\section{More Identities}

We will state and prove some more necessary identities.

\subsection{Identities for $(x_\alpha^+\otimes b)p_i(\varphi,\chi)$ and $p_i(\varphi,\chi)(x_\alpha^-\otimes b)$}

\begin{prop}\label{x+pastq}
For all $\varphi,\chi\in\f$, $\alpha\in R^+$, $i\in I$ and
$b\in A$ with $\alpha(h_i)\neq0$ the
following hold:\\
\noindent$(i)$
$$(x_\alpha^+\otimes b)p_i(\varphi,\chi)
=\sum_{\substack{\psi_1\in\f(\varphi)\\ \psi_2\in\f(\chi)\\ |\psi_1|=|\psi_2|}}
\binom{\alpha(h_i)+|\psi_1|-1}{|\psi_1|}\m(\psi_1)\m(\psi_2)p_i(\varphi-\psi_1,\chi-\psi_2)
\left(x_\alpha^+\otimes b\pi(\psi_1)\pi(\psi_2)\right)$$
\noindent$(ii)$
$$p_i(\varphi,\chi)(x_\alpha^-\otimes b)
=\sum_{\substack{\psi_1\in\f(\varphi)\\ \psi_2\in\f(\chi)\\ |\psi_1|=|\psi_2|}}
\binom{\alpha(h_i)+|\psi_1|-1}{|\psi_1|}\m(\psi_1)\m(\psi_2)\left(x_\alpha^-\otimes b\pi(\psi_1)\pi(\psi_2)\right)
p_i(\varphi-\psi_1,\chi-\psi_2)$$
\end{prop}

The proof proceeds by induction on $|\chi|=|\varphi|$. Details
can be found in \cite{Ch}.

\subsection{Identities for $(x_\alpha^+\otimes b)^{(r)}p_i(\chi)$
and $p_i(\chi)(x_\alpha^-\otimes b)^{(s)}$}

\noindent\textbf{Definition.} Given $\chi\in\f$ define
$$\cs(\chi)=\left\{\psi\in\f(\f):\sum_{\phi\in\f}\psi(\phi)\phi\leq\chi\right\}$$
and $\cs_k(\chi)=\cs(\chi)\cap\f_k(\f)$.

\begin{prop}\label{x+rpastq}
For all $\chi\in\f$, $\alpha\in R^+$, $i\in I$ and $b\in A$
with $\alpha(h_i)\neq0$ and $r,s\in\bbz_{\geq0}$ the following
hold:\\
\noindent$(i)$
\begin{eqnarray*}
&&\hskip-.5in(x_\alpha^+\otimes b)^{(r)}p_i(\chi)\\
&=&\sum_{\psi\in\cs_r(\chi)}p_i\left(\chi-\sum_{\phi\in\f}\psi(\phi)\phi\right)
\prod_{\phi\in\f}\left(\binom{\alpha(h_i)+|\phi|-1}{|\phi|}\m(\phi)
\left(x_\alpha^+\otimes b\pi(\phi)\right)\right)^{(\psi(\phi))}
\end{eqnarray*}
\noindent$(ii)$
\begin{eqnarray*}
&&\hskip-.5inp_i(\chi)(x_\alpha^-\otimes b)^{(r)}\\
&=&\sum_{\psi\in\cs_r(\chi)}\prod_{\phi\in\f}\left(\binom{\alpha(h_i)+|\phi|-1}{|\phi|}\m(\phi)
\left(x_\alpha^-\otimes b\pi(\phi)\right)\right)^{(\psi(\phi))}p_i\left(\chi-\sum_{\phi\in\f}\psi(\phi)\phi\right)
\end{eqnarray*}
\end{prop}

The proof proceeds by induction on $r$. The case $r=1$ is
Proposition \ref{x+pastq}. We also use Proposition
\ref{x+pastq} at the beginning of the proof. See \cite{Ch} for
details.

\section{A Triangular Decomposition of $\bu_\bbz(\lie g\otimes A)$}

This section is devoted to the proof of the following lemma.
\begin{lem}\label{tri}
$$\bu_\bbz(\lie g\otimes A)=\bu_\bbz^-(\lie g\otimes A)\bu^0_\bbz(\lie g\otimes
A)\bu^+_\bbz(\lie g\otimes A)$$
\end{lem}

\noindent\textbf{Definition.} Define a \emph{monomial} in
$\bu_\bbz(\lie g\otimes A)$ to be any product of elements in
the set $$\{(x_\alpha^\pm\otimes b)^{(r)},\ q_i(\chi):\alpha\in
R^+,\ b\in\bb,\ r\in\bbz_{\geq0},\ i\in I,\ \chi\in\f(\bb)\}$$

By induction on the degree of monomials in $\bu_\bbz(\lie
g\otimes A)$ Lemma \ref{tri} will follow from the next lemma.

\subsection{Brackets in $\bu_\bbz(\lie g\otimes A)$}

\begin{lem}
For all $\alpha\in R^+$, $i\in I$, $a,b\in\bb$,
$\chi\in\f(\bb)$ and $r,s\in\bbz_{\geq0}$ the following hold\\
\noindent$(a)$ $\displaystyle\left[(x_\alpha^+\otimes
a)^{(r)},(x_\alpha^-\otimes
b)^{(s)}\right]\in\bbz-\span\mathcal{B}$ and
has degree less than $r+s$.\\
\noindent$(b)$ $\displaystyle\left[(x_\alpha^+\otimes
a)^{(r)},p_i(\chi)\right]\in\bbz-\span\mathcal{B}$ and has
degree less than $r+|\chi|$.\\
\noindent$(c)$ $\displaystyle\left[p_i(\chi),(x_\alpha^-\otimes
a)^{(s)}\right]\in\bbz-\span\mathcal{B}$ and has degree less
than $s+|\chi|$.
\end{lem}
\begin{proof}
We start by proving $(a)$. By Lemma \ref{basic}, for all
$a,b\in\bb$ and $r,s\in\bbz_{\geq0}$,
\begin{eqnarray*}
&&\hskip-.4in x^+(r\chi_a)x^-(s\chi_b)\\
&=&\sum_{\substack{j,k\in\bbz_{\geq0}\\j+k\leq\min(r,s)}}\sum_{l=0}^k(-1)^{j+k}D^-(j\chi_a,j\chi_b,(s-j-k)\chi_b)
p(l\chi_a,l\chi_b)\\
&\times&D^+((k-l)\chi_a,(k-l)\chi_b,(r-j-k)\chi_a)
\end{eqnarray*}
It can be easily shown by induction on $l$ that
$q\left(l\chi_a,l\chi_b\right)=q\left(l\chi_{ab}\right)$. So we
have
\begin{eqnarray*}
x^+(r\chi_a)x^-(s\chi_b)
&=&\sum_{\substack{j,k\in\bbz_{\geq0}\\0<j+k\leq\min(r,s)}}\sum_{l=0}^k(-1)^{j+k}D^-(j\chi_a,j\chi_b,(s-j-k)\chi_b)
p(l\chi_{ab})\\
&\times&D^+((k-l)\chi_a,(k-l)\chi_b,(r-j-k)\chi_a)\\
&+&x^-(s\chi_b)x^+(r\chi_a)
\end{eqnarray*}
$D^-(j\chi_a,j\chi_b,(s-j-k)\chi_b)\in\bbz-\span\mathcal{B}_-$
and
$D^+((k-l)\chi_a,(k-l)\chi_b,(r-j-k)\chi_a)\bbz-\span\mathcal{B}_+$
by Proposition \ref{Dpmexp}. Also we clearly see that the
degree of this commutator is less than $r+s$. Thus $(a)$ is
proved. $(b)$ and $(c)$ hold by \eqref{binom} and Proposition
\ref{x+rpastq}.
\end{proof}

The theorem is now proved once we prove Lemma \ref{basic}. The
next section are dedicated to the proof of Lemma \ref{basic}.

\section{Proof of Lemma \ref{basic}}

We will adapt and extend the proof of the corresponding result
in \cite{CP}. We begin with some necessary identities.

\subsection{Identity for $p(\varphi,\chi)(x^+\otimes b)$}

\begin{lem}\label{qpastx+}
For all $\varphi,\chi\in\f$, and $b\in A$.
\begin{eqnarray*}
p(\varphi,\chi)(x^+\otimes b)
&=&(x^+\otimes b)p(\varphi,\chi)-2\sum_{\substack{c\in\supp\varphi\\ d\in\supp\chi}}(x^+\otimes bcd)
p(\varphi-\chi_c,\chi-\chi_d)\\
&+&\sum_{\substack{\phi_1\in\f_2(\varphi)\\ \phi_2\in\f_2(\chi)}}\m(\phi_1)\m(\phi_2)
\left(x^+\otimes b\pi(\psi_1)\pi(\psi_2)\right)p(\varphi-\psi_1,\chi-\psi_2)
\end{eqnarray*}
\end{lem}

The proof proceeds by induction on $|\chi|$. Details can be
found in \cite{Ch}.

\subsection{Identities Involving $D^\pm(\psi_1,\psi_2,\psi_3)$}

\begin{prop}\label{idD}
Let $b\in A$, and $\psi_1,\psi_2,\psi_3\in\f$ with $|\psi_3|\geq1$ be given. Then\\
\noindent$(i)$
\begin{eqnarray*}
\psi_2(b)D^\pm(\psi_1,\psi_2,\psi_3)&=&\sum_{\substack{\phi_1\in\f(\psi_1)\\ \phi_2\in\f(\psi_2)}}
\sum_{c\in\supp\psi_3}\phi_2(b)D^\pm(\phi_1,\phi_2,\chi_c)D^\pm(\psi_1-\phi_1,\psi_2-\phi_2,\psi_3-\chi_c)
\end{eqnarray*}
\noindent$(ii)$
\begin{eqnarray*}
&&\hskip-.6in(|\psi_2|+|\psi_3|)D^\pm(\psi_1,\psi_2,\psi_3)\\
&=&\sum_{\substack{\phi_1\in\f(\psi_1)\\ \phi_2\in\f(\psi_2)}}\sum_{c\in\supp\psi_3}(|\phi_1|+1)
D^\pm(\phi_1,\phi_2,\chi_c)D^\pm(\psi_1-\phi_1,\psi_2-\phi_2,\psi_3-\chi_c)
\end{eqnarray*}
\end{prop}

The proofs proceed by induction on $|\psi_3|$. Details can
found in \cite{Ch}.

The following lemma is necessary for the proof of Lemma
\ref{basic}.

\subsection{Identity for $\sum_{\phi\in\f_{|\chi|}}\bbd(\phi,\chi,\varphi-\phi)(x^-\otimes
b)$}

\begin{lem}\label{idbbd}
For all $b\in A$ and $\varphi,\chi\in\f$,
\begin{eqnarray*}
&&\hskip-.6in\sum_{\phi\in\f(\varphi)}\bbd(\phi,\chi,\varphi-\phi)(x^-\otimes b)\\
&=&-(\chi(b)+1)\sum_{\phi'\in\f(\varphi)}\bbd(\phi',\chi+\chi_b,\varphi-\phi')\\
&+&\sum_{\phi\in\f(\varphi)}\sum_{\substack{\phi_1\in\f(\phi)\\ \phi_2\in\f(\chi)}}(|\phi_1|+1)
D^-(\phi_1,\phi_2,\chi_b)\bbd(\phi-\phi_1,\chi-\phi_2,\varphi-\phi)
\end{eqnarray*}
\end{lem}

\subsection{Proof of Lemma \ref{basic}}

The following is the proof of Lemma \ref{basic} using Lemma
\ref{idbbd} and induction on $|\chi|$.
\begin{proof}
It can easily be shown by induction on $|\varphi|$ that
$$\bbd(0,0,\varphi)=x^+(\varphi)$$
Thus Lemma \ref{basic} is true for $\chi=0$. Assume that
$\chi\in\f-\{0\}$ then
\begin{eqnarray*}
|\chi|x^+(\varphi)x^-(\chi)
&=&\sum_{b\in\supp\chi}\chi(b)x^+(\varphi)x^-(\chi)\\
&=&\sum_{b\in\supp\chi}x^+(\varphi)x^-(\chi-\chi_b)(x^-\otimes b)\\
&=&\sum_{b\in\supp\chi}
\sum_{\substack{\psi_1,\psi_2,\phi_1,\phi_2\in\f\\ \psi_1+\psi_2\in\f(\chi-\chi_b)\\ \phi_1+\phi_2\in\f(\varphi)}}
(-1)^{|\psi_1|+|\psi_2|}D^-(\phi_1,\psi_1,\chi-\chi_b-\psi_1-\psi_2)\\
&\times&\bbd(\phi_2,\psi_2,\varphi-\phi_1-\phi_2)(x^-\otimes b)
\hskip.5in(\textnormal{by the induction hypothesis})\\
&=&\sum_{b\in\supp\chi}
\sum_{\substack{\psi_1,\psi_2,\phi_1\in\f\\ \psi_1+\psi_2\in\f(\chi-\chi_b)\\ \phi_1\in\f(\varphi)}}
(-1)^{|\psi_1|+|\psi_2|}D^-(\phi_1,\psi_1,\chi-\chi_b-\psi_1-\psi_2)\\
&\times&\sum_{\substack{\phi_2\in\f(\varphi-\phi_1)}}\bbd(\phi_2,\psi_2,\varphi-\phi_1-\phi_2)(x^-\otimes b)\\
&=&\sum_{b\in\supp\chi}
\sum_{\substack{\psi_1,\psi_2,\phi_1\in\f\\ \psi_1+\psi_2\in\f(\chi-\chi_b)\\ \phi_1\in\f(\varphi)}}
(-1)^{|\psi_1|+|\psi_2|}D^-(\phi_1,\psi_1,\chi-\chi_b-\psi_1-\psi_2)\\
&\times&\Bigg(-(\psi_2(b)+1)\sum_{\phi_2'\in\f(\varphi-\phi_1)}\bbd(\phi_2',\psi_2+\chi_b,\varphi-\phi_1-\phi_2')\\
&+&\sum_{\substack{\phi_2\in\f(\varphi-\phi_1)}}
\sum_{\substack{\tau\in\f(\phi_2)\\ \tau'\in\f(\psi_2)}}(|\tau|+1)D^-(\tau,\tau',\chi_b)
\bbd(\phi_2-\tau,\psi_2-\tau',\varphi-\phi_1-\phi_2)\Bigg)\\
&&(\textnormal{by Lemma \ref{idbbd}})
\end{eqnarray*}

\begin{eqnarray*}
&=&\sum_{b\in\supp\chi}
\sum_{\substack{\psi_1,\psi_2,\phi_1\in\f\\ \psi_1+\psi_2\in\f(\chi-\chi_b)\\ \phi_1\in\f(\varphi)}}
(-1)^{|\psi_1|+|\psi_2|+1}D^-(\phi_1,\psi_1,\chi-\chi_b-\psi_1-\psi_2)(\psi_2(b)+1)\\
&\times&\sum_{\phi_2'\in\f(\varphi-\phi_1)}\bbd(\phi_2',\psi_2+\chi_b,\varphi-\phi_1-\phi_2')\\
&+&\sum_{b\in\supp\chi}
\sum_{\substack{\psi_1,\psi_2,\phi_1\in\f\\ \psi_1+\psi_2\in\f(\chi-\chi_b)\\ \phi_1\in\f(\varphi)}}
(-1)^{|\psi_1|+|\psi_2|}D^-(\phi_1,\psi_1,\chi-\chi_b-\psi_1-\psi_2)\\
&\times&\sum_{\substack{\phi_2\in\f(\varphi-\phi_1)}}\sum_{\substack{\tau\in\f(\phi_2)\\ \tau'\in\f(\psi_2)}}(|\tau|+1)
D^-(\tau,\tau',\chi_b)\bbd(\phi_2-\tau,\psi_2-\tau',\varphi-\phi_1-\phi_2)\\
&=&\sum_{b\in\supp\chi}\sum_{\substack{\psi_1,\psi_2',\phi_1\in\f\\ \psi_1+\psi_2'\in\f(\chi)\\ \phi_1\in\f(\varphi)}}
(-1)^{|\psi_1|+|\psi_2'|}D^-(\phi_1,\psi_1,\chi-\psi_1-\psi_2')\psi_2'(b)\\
&\times&\sum_{\phi_2'\in\f(\varphi-\phi_1)}\bbd(\phi_2',\psi_2',\varphi-\phi_1-\phi_2')\\
&+&\sum_{\substack{\psi_1,\psi_2,\phi_1,\phi_2\in\f\\ \psi_1+\psi_2\in\f(\chi)\\ \phi_1+\phi_2\in\f(\varphi)}}
\sum_{b\in\supp(\chi-\psi_1-\psi_2)}\sum_{\substack{\tau\in\f(\phi_2)\\ \tau'\in\f(\psi_2)}}
(|\tau|+1)(-1)^{|\psi_1|+|\psi_2|}D^-(\phi_1,\psi_1,\chi-\chi_b-\psi_1-\psi_2)\\
&\times&D^-(\tau,\tau',\chi_b)\bbd(\phi_2-\tau,\psi_2-\tau',\varphi-\phi_1-\phi_2)\\
&=&\sum_{\substack{\psi_1,\psi_2',\phi_1,\phi_2'\in\f\\ \psi_1+\psi_2'\in\f(\chi)\\ \phi_1+\phi_2'\in\f(\varphi)}}
(-1)^{|\psi_1|+|\psi_2'|}|\psi_2'|D^-(\phi_1,\psi_1,\chi-\psi_1-\psi_2')\bbd(\phi_2',\psi_2',\varphi-\phi_1-\phi_2')\\
&+&\sum_{\substack{\psi_1',\psi_2',\phi_1',\phi_2'\in\f\\ \psi_1'+\psi_2'\in\f(\chi)\\ \phi_1'+\phi_2'\in\f(\varphi)}}
(-1)^{|\psi_1'|+|\psi_2'|}\sum_{\substack{\tau\in\f(\phi_1')\\ \tau'\in\f(\psi_1')}}
\sum_{b\in\supp(\chi-\psi_1'-\psi_2')}(|\tau|+1)D^-(\tau,\tau',\chi_b)\\
&\times&D^-(\phi_1'-\tau,\psi_1'-\tau',\chi-\chi_b-\psi_1'-\psi_2')
\bbd(\phi_2',\psi_2',\varphi-\phi_1-\phi_2)\\
&=&\sum_{\substack{\psi_1,\psi_2',\phi_1,\phi_2'\in\f\\ \psi_1+\psi_2'\in\f(\chi)\\ \phi_1+\phi_2'\in\f(\varphi)}}
(-1)^{|\psi_1|+|\psi_2'|}|\psi_2'|D^-(\phi_1,\psi_1,\chi-\psi_1-\psi_2')\bbd(\phi_2',\psi_2',\varphi-\phi_1-\phi_2')\\
&+&\sum_{\substack{\psi_1',\psi_2',\phi_1',\phi_2'\in\f\\ \psi_1'+\psi_2'\in\f(\chi)\\ \phi_1'+\phi_2'\in\f(\varphi)}}
(-1)^{|\psi_1'|+|\psi_2'|}(|\chi|-|\psi_2'|)D^-(\phi_1',\psi_1',\chi-\psi_1'-\psi_2')
\bbd(\phi_2',\psi_2',\varphi-\phi_1-\phi_2)\\
&&\hskip-.3in(\textnormal{by Proposition }\ref{idD}(ii))
\end{eqnarray*}

\begin{eqnarray*}
&=&|\chi|\sum_{\substack{\psi_1,\psi_2,\phi_1,\phi_2\in\f\\ \psi_1+\psi_2\in\f(\chi)\\ \phi_1+\phi_2\in\f(\varphi)}}
(-1)^{|\psi_1|+|\psi_2|}D^-(\phi_1,\psi_1,\chi-\psi_1-\psi_2)\bbd(\phi_2,\psi_2,\varphi-\phi_1-\phi_2)
\end{eqnarray*}
\end{proof}
So all that remains to prove for Lemma \ref{basic} and Theorem
\ref{thm} is Lemma \ref{idbbd}.

\subsection{Proof of Lemma \ref{idbbd}}

\begin{proof}
If $|\phi|>|\varphi|$ we have 0 on both sides. So the identity
is trivially true in this case. Assume that
$|\varphi|\geq|\phi|$ and proceed by induction on
$|\varphi|-|\phi|$. In the case $|\phi|=|\varphi|$ Lemma
\ref{idbbd} becomes $$p(\varphi,\chi)(x^-\otimes b)
=\sum_{\substack{\phi_1\in\f(\varphi)\\
\phi_2\in\f(\chi)}}(\phi_1+1)\m(\phi_1)\m(\phi_2)(x^-\otimes
b\pi(\phi_1)\pi(\phi_2)) p(\varphi-\phi_1,\chi-\phi_2)$$ which
is just Lemma \ref{x+pastq}$(ii)$. So Assume that
$|\varphi|-|\phi|>0$ then by Lemma \ref{qpastx+} we have
\begin{eqnarray}\label{eqnbbd}
(|\varphi|-|\phi|)\bbd(\phi,\chi,\varphi-\phi)
&=&\sum_{c\in\supp(\varphi-\phi)}\Bigg((x^+\otimes c)\bbd(\phi,\chi,\varphi-\phi-\chi_c)\\
&-&\sum_{\substack{d\in\supp\phi\\ d'\in\supp\chi}}(x^+\otimes cdd')
\bbd(\phi-\chi_d,\chi-\chi_{d'},\varphi-\phi-\chi_c)\Bigg)\nonumber
\end{eqnarray}
Hence
\begin{eqnarray*}
&&\hskip-.6in(|\varphi|-|\phi|)\sum_{\phi\in\f(\varphi)}\bbd(\phi,\chi,\varphi-\phi)(x^-\otimes b)\\
&=&\sum_{\phi\in\f(\varphi)}\sum_{c\in\supp(\varphi-\phi)}
\Bigg((x^+\otimes c)\bbd(\phi,\chi,\varphi-\phi-\chi_c)\\
&-&\sum_{\substack{d\in\supp\phi\\ d'\in\supp\chi}}(x^+\otimes cdd')
\bbd(\phi-\chi_d,\chi-\chi_{d'},\varphi-\phi-\chi_c)\Bigg)(x^-\otimes b)\\
&=&\sum_{c\in\supp\varphi}(x^+\otimes c)\sum_{\phi\in\f(\varphi-\chi_c)}
\bbd(\phi,\chi,\varphi-\phi-\chi_c)(x^-\otimes b)\\
&-&\sum_{\substack{c\in\supp\varphi\\ d\in\supp(\varphi-\chi_c)\\ d'\in\supp\chi}}(x^+\otimes cdd')
\sum_{\phi'\in\f(\varphi-\chi_c-\chi_d)}\bbd(\phi',\chi-\chi_{d'},\varphi-\phi'-\chi_d-\chi_c)(x^-\otimes b)\\
\end{eqnarray*}
\begin{eqnarray*}
&=&-(\chi(b)+1)\sum_{c\in\supp\varphi}(x^+\otimes c)\sum_{\phi'\in\f(\varphi-\chi_c)}
\bbd(\phi',\chi+\chi_b,\varphi-\phi'-\chi_c)\\
&+&\sum_{c\in\supp\varphi}\sum_{\phi\in\f(\varphi-\chi_c)}
\sum_{\substack{\phi_1\in\f(\phi)\\ \phi_2\in\f(\chi)}}(|\phi_1|+1)(x^+\otimes c)D^-(\phi_1,\phi_2,\chi_b)
\bbd(\phi-\phi_1,\chi-\phi_2,\varphi-\phi-\chi_c)\\
&+&\sum_{\substack{d'\in\supp\chi\\ c\in\supp\varphi\\ d\in\supp(\varphi-\chi_c)}}
\sum_{\phi'\in\f(\varphi-\chi_c-\chi_d)}(\chi(b)-\chi_{d'}(b)+1)(x^+\otimes cdd')\\
&\times&\bbd(\phi',\chi-\chi_{d'}+\chi_b,\varphi-\phi'-\chi_d-\chi_c)\\
&-&\sum_{\substack{d'\in\supp\chi\\ c\in\supp\varphi\\ d\in\supp(\varphi-\chi_c)}}
\sum_{\phi\in\f(\varphi-\chi_c-\chi_d)}\sum_{\substack{\phi_1\in\f(\phi)\\ \phi_2\in\f(\chi-\chi_{d'})}}(|\phi_1|+1)
(x^+\otimes cdd')D^-(\phi_1,\phi_2,\chi_b)\\
&\times&\bbd(\phi-\phi_1,\chi-\chi_{d'}-\phi_2,\varphi-\phi-\chi_d-\chi_c)
\hskip.2in(\textnormal{by the induction hypothesis})\\
&=&-(\chi(b)+1)\sum_{\phi'\in\f(\varphi)}\sum_{c\in\supp(\varphi-\phi')}
(x^+\otimes c)\bbd(\phi',\chi+\chi_b,\varphi-\phi'-\chi_c)\\
&+&\sum_{\phi\in\f(\varphi)}\sum_{c\in\supp(\varphi-\phi)}
\sum_{\substack{\phi_1\in\f(\phi)\\ \phi_2\in\f(\chi)}}(|\phi_1|+1)(x^+\otimes c)D^-(\phi_1,\phi_2,\chi_b)
\bbd(\phi-\phi_1,\chi-\phi_2,\varphi-\phi-\chi_c)\\
&+&\sum_{\phi\in\f(\varphi)}\sum_{c\in\supp(\varphi-\phi)}
\sum_{\substack{d\in\supp\phi\\ d'\in\supp\chi}}(\chi(b)-\chi_{d'}(b)+1)(x^+\otimes cdd')\\
&\times&\bbd(\phi-\chi_d,\chi-\chi_{d'}+\chi_b,\varphi-\phi-\chi_c)\\
&-&\sum_{\phi\in\f(\varphi)}\sum_{c\in\supp(\varphi-\phi)}\sum_{\substack{d\in\supp\phi\\ d'\in\supp\chi}}
\sum_{\substack{\phi_1\in\f(\phi-\chi_d)\\ \phi_2\in\f(\chi-\chi_{d'})}}(|\phi_1|+1)
(x^+\otimes cdd')D^-(\phi_1,\phi_2,\chi_b)\\
&\times&\bbd(\phi-\chi_d-\phi_1,\chi-\chi_{d'}-\phi_2,\varphi-\phi-\chi_c)
\end{eqnarray*}
On the other hand
\begin{eqnarray*}
&&\hskip-.6in(|\varphi|-|\phi|)\sum_{\phi\in\f(\varphi)}
\sum_{\substack{\phi_1\in\f(\phi)\\ \phi_2\in\f(\chi)}}(|\phi_1|+1)D^-(\phi_1,\phi_2,\chi_b)
\bbd(\phi-\phi_1,\chi-\phi_2,\varphi-\phi)\\
&=&\sum_{\phi\in\f(\varphi)}\sum_{\substack{\phi_1\in\f(\phi)\\ \phi_2\in\f(\chi)}}
(|\phi_1|+1)D^-(\phi_1,\phi_2,\chi_b)\sum_{c\in\supp(\varphi-\phi)}(x^+\otimes c)
\bbd(\phi-\phi_1,\chi-\phi_2,\varphi-\phi-\chi_c)\\
&-&\sum_{\phi\in\f(\varphi)}\sum_{\substack{\phi_1\in\f(\phi)\\ \phi_2\in\f(\chi)}}
(|\phi_1|+1)D^-(\phi_1,\phi_2,\chi_b)\sum_{c\in\supp(\varphi-\phi)}
\sum_{\substack{d\in\supp\phi-\phi_1\\ d'\in\supp\chi-\phi_2}}(x^+\otimes cdd')\\
&\times&\bbd(\phi-\phi_1-\chi_d,\chi-\phi_2-\chi_{d'},\varphi-\phi-\chi_c)\hskip.2in(\textnormal{by \eqref{eqnbbd}})
\end{eqnarray*}

So it suffices to show that
\begin{eqnarray*}
0&=&-(\chi(b)+1)\sum_{\phi'\in\f(\varphi)}\sum_{c\in\supp(\varphi-\phi')}(x^+\otimes c)
\bbd(\phi',\chi+\chi_b,\varphi-\phi'-\chi_c)\\
&+&\sum_{\phi\in\f(\varphi)}\sum_{c\in\supp(\varphi-\phi)}\sum_{\substack{\phi_1\in\f(\phi)\\ \phi_2\in\f(\chi)}}
(|\phi_1|+1)\left[(x^+\otimes c),D^-(\phi_1,\phi_2,\chi_b)\right]\\
&\times&\bbd(\phi-\phi_1,\chi-\phi_2,\varphi-\phi-\chi_c)\\
&+&\sum_{\phi'\in\f(\varphi)}\sum_{c\in\supp(\varphi-\phi')}\sum_{\substack{d\in\supp\phi'\\ d'\in\supp\chi}}
(\chi(b)-\chi_{d'}(b)+1)(x^+\otimes cdd')\\
&\times&\bbd(\phi'-\chi_d,\chi-\chi_{d'}+\chi_b,\varphi-\phi'-\chi_c)\\
&-&\sum_{\phi\in\f(\varphi)}\sum_{c\in\supp(\varphi-\phi)}\sum_{\substack{d\in\supp\phi\\ d'\in\supp\chi}}
\sum_{\substack{\phi_1\in\f(\phi-\chi_d)\\ \phi_2\in\f(\chi-\chi_{d'})}}(|\phi_1|+1)
\left[(x^+\otimes cdd'),D^-(\phi_1,\phi_2,\chi_b)\right]\\
&\times&\bbd(\phi-\chi_d-\phi_1,\chi-\chi_{d'}-\phi_2,\varphi-\phi-\chi_c)\\
&+&(|\varphi|-|\phi|)(\chi(b)+1)\sum_{\phi'\in\f(\varphi)}\bbd(\phi',\chi+\chi_b,\varphi-\phi')
\end{eqnarray*}

Concentrating on the terms with Lie brackets we obtain
\begin{eqnarray*}
&&\hskip-.6in\sum_{\phi\in\f(\varphi)}\sum_{c\in\supp(\varphi-\phi)}
\sum_{\substack{\phi_1\in\f(\phi)\\ \phi_2\in\f(\chi)}}(|\phi_1|+1)\left[(x^+\otimes c),D^-(\phi_1,\phi_2,\chi_b)\right]
\bbd(\phi-\phi_1,\chi-\phi_2,\varphi-\phi-\chi_c)\\
&-&\sum_{\phi\in\f(\varphi)}\sum_{c\in\supp(\varphi-\phi)}\sum_{\substack{d\in\supp\phi\\ d'\in\supp\chi}}
\sum_{\substack{\phi_1\in\f(\phi-\chi_d)\\ \phi_2\in\f(\chi-\chi_{d'})}}(|\phi_1|+1)
\left[(x^+\otimes cdd'),D^-(\phi_1,\phi_2,\chi_b)\right]\\
&\times&\bbd(\phi-\chi_d-\phi_1,\chi-\chi_{d'}-\phi_2,\varphi-\phi-\chi_c)\\
&=&\sum_{\phi'\in\f(\varphi)}\sum_{\substack{\psi_1'\in\f(\phi')\\ \psi_2\in\f(\chi)}}\sum_{c\in\supp\psi_1'}
\sum_{\substack{\phi_1\in\f(\psi_1'-\chi_c)\\ \phi_2\in\f(\psi_2)\\ |\phi_1|=|\phi_2|}}\m(\phi_1)\m(\phi_2)
(h\otimes bc\pi(\phi_1)\pi(\phi_2))p(\psi_1'-\chi_c-\phi_1,\psi_2-\phi_2)\\
&\times&D^+(\phi'-\psi_1',\chi-\psi_2,\varphi-\phi')
\end{eqnarray*}
The following formula will be necessary, it can be proved by
induction on $|\varphi|$ (see \cite{Ch} for details). Suppose
that $b\in A$ and $\varphi,\chi\in\f$ then
\begin{eqnarray}\label{eqnq}
&&\hskip-.6in-(\chi(b)+1)p(\varphi,\chi+\chi_b)\nonumber\\
&=&\sum_{c\in\supp\varphi}\sum_{\substack{\phi_1\in\f(\varphi-\chi_c)\\ \phi_2\in\f(\chi)\\ |\phi_1|=|\phi_2|}}
\m(\phi_1)\m(\phi_2)(h\otimes bc\pi(\phi_1)\pi(\phi_2))p(\varphi-\chi_c-\phi_1,\chi-\phi_2)
\end{eqnarray}
From above
\begin{eqnarray*}
&&\hskip-.6in\sum_{\phi\in\f(\varphi)}\sum_{c\in\supp(\varphi-\phi)}
\sum_{\substack{\phi_1\in\f(\phi)\\ \phi_2\in\f(\chi)}}(|\phi_1|+1)\left[(x^+\otimes c),D^-(\phi_1,\phi_2,\chi_b)\right]
\bbd(\phi-\phi_1,\chi-\phi_2,\varphi-\phi-\chi_c)\\
&-&\sum_{\phi\in\f(\varphi)}\sum_{c\in\supp(\varphi-\phi)}\sum_{\substack{d\in\supp\phi\\ d'\in\supp\chi}}
\sum_{\substack{\phi_1\in\f(\phi-\chi_d)\\ \phi_2\in\f(\chi-\chi_{d'})}}(|\phi_1|+1)
\left[(x^+\otimes cdd'),D^-(\phi_1,\phi_2,\chi_b)\right]\\
&\times&\bbd(\phi-\chi_d-\phi_1,\chi-\chi_{d'}-\phi_2,\varphi-\phi-\chi_c)\\
&=&\sum_{\phi'\in\f(\varphi)}\sum_{\substack{\psi_1'\in\f(\phi')\\ \psi_2\in\f(\chi)}}\sum_{c\in\supp\psi_1'}
\sum_{\substack{\phi_1\in\f(\psi_1'-\chi_c)\\ \phi_2\in\f(\psi_2)\\ |\phi_1|=|\phi_2|}}\m(\phi_1)\m(\phi_2)
(h\otimes bc\pi(\phi_1)\pi(\phi_2))p(\psi_1'-\chi_c-\phi_1,\psi_2-\phi_2)\\
&\times&D^+(\phi'-\psi_1',\chi-\psi_2,\varphi-\phi')\\
&=&-\sum_{\phi'\in\f(\varphi)}\sum_{\substack{\psi_1'\in\f(\phi')\\ \psi_2\in\f(\chi)}}(\psi_2(b)+1)
p(\psi_1',\psi_2+\chi_b)D^+(\phi'-\psi_1',\chi-\psi_2,\varphi-\phi')\hskip.2in(\textnormal{by \eqref{eqnq}})\\
&=&-\sum_{\phi'\in\f(\varphi)}\sum_{\substack{\psi_1'\in\f(\phi')\\ \psi_2'\in\f(\chi+\chi_b)}}\psi_2'(b)
p(\psi_1',\psi_2')D^+(\phi'-\psi_1',\chi+\chi_b-\psi_2',\varphi-\phi')
\end{eqnarray*}
So the proof of Lemma \ref{idbbd} is reduced to showing
\begin{eqnarray}\label{Dtobbd}
0&=&-\sum_{\phi'\in\f(\varphi)}\sum_{\substack{\psi_1'\in\f(\phi')\\ \psi_2'\in\f(\chi+\chi_b)}}\psi_2'(b)
p(\psi_1',\psi_2')D^+(\phi'-\psi_1',\chi+\chi_b-\psi_2',\varphi-\phi')\nonumber\\
&-&(\chi(b)+1)\sum_{\phi'\in\f(\varphi)}\sum_{c\in\supp(\varphi-\phi')}(x^+\otimes c)
\bbd(\phi',\chi+\chi_b,\varphi-\phi'-\chi_c)\nonumber\\
&+&\sum_{\phi'\in\f(\varphi)}\sum_{c\in\supp(\varphi-\phi')}\sum_{\substack{d\in\supp\phi'\\ d'\in\supp\chi}}
(\chi(b)-\chi_{d'}(b)+1)(x^+\otimes cdd')\nonumber\\
&\times&\bbd(\phi'-\chi_d,\chi-\chi_{d'}+\chi_b,\varphi-\phi'-\chi_c)\nonumber\\
&+&(|\varphi|-|\phi|)(\chi(b)+1)\sum_{\phi'\in\f(\varphi)}\bbd(\phi',\chi+\chi_b,\varphi-\phi')
\end{eqnarray}
Using Proposition \ref{x+pastq}$(i)$ we have
\begin{eqnarray*}
&&\hskip-.6in(x^+\otimes c)\bbd(\phi',\chi+\chi_b,\varphi-\phi'-\chi_c)\\
&=&\sum_{\substack{\phi_1\in\f(\phi')\\ \phi_2\in\f(\chi+\chi_b)}}p(\phi_1,\phi_2)
\sum_{\substack{\phi_1'\in\f(\phi'-\phi_1)\\ \phi_2'\in\f(\chi+\chi_b-\phi_2)\\|\phi_1'|=|\phi_2'|}}(|\phi_1'|+1)
\m\left(\phi_1'\right)\m\left(\phi_2'\right)\left(x^+\otimes c\pi(\phi_1')\pi(\phi_2')\right)\\
&\times&D^+(\phi'-\phi_1'-\phi_1,\chi+\chi_b-\phi_2'-\phi_2,\varphi-\phi'-\chi_c)
\end{eqnarray*}
Similarly
\begin{eqnarray*}
&&\hskip-.6in(x^+\otimes cdd')\bbd(\phi'-\chi_d,\chi+\chi_b-\chi_{d'},\varphi-\phi'-\chi_c)\\
&=&\sum_{\substack{\phi_1\in\f(\phi')\\ \phi_2\in\f(\chi+\chi_b)}}p(\phi_1,\phi_2)
\sum_{\substack{\varphi_1\in\f(\phi'-\phi_1)\\ \varphi_2\in\f(\chi+\chi_b-\phi_2)\\|\varphi_1|=|\varphi_2|}}
\frac{\m(\varphi_1)}{|\varphi_1|}\m(\varphi_2)\varphi_1(d)\varphi_2(d')(x^+\otimes c\pi(\varphi_1)\pi(\varphi_2))\\
&\times&D^+(\phi'-\phi_1-\varphi_1,\chi+\chi_b-\phi_2-\varphi_2,\varphi-\phi'-\chi_c)
\end{eqnarray*}
Using these identities \eqref{Dtobbd} becomes
\begin{eqnarray*}
&&\hskip-.6in\sum_{\phi'\in\f(\varphi)}\sum_{\substack{\phi_1\in\f(\phi')\\ \phi_2\in\f(\chi+\chi_b)}}
p(\phi_1,\phi_2)\phi_2(b)D^+(\phi'-\phi_1,\chi+\chi_b-\phi_2,\varphi-\phi')\\
&=&\sum_{\phi'\in\f(\varphi)}\sum_{\substack{\phi_1\in\f(\phi')\\ \phi_2\in\f(\chi+\chi_b)}}p(\phi_1,\phi_2)
(-(\chi(b)+1))\sum_{c\in\supp(\varphi-\phi')}
\sum_{\substack{\phi_1'\in\f(\phi'-\phi_1)\\ \phi_2'\in\f(\chi+\chi_b-\phi_2)\\ |\phi_1'|=|\phi_2'|}}(|\phi_1'|+1)
\m\left(\phi_1'\right)\m\left(\phi_2'\right)\\
&\times&(x^+\otimes c\pi(\phi_1')\pi(\phi_2'))
D^+(\phi'-\phi_1'-\phi_1,\chi+\chi_b-\phi_2'-\phi_2,\varphi-\phi'-\chi_c)\\
&+&\sum_{\phi'\in\f(\varphi)}\sum_{\substack{\phi_1\in\f(\phi')\\ \phi_2\in\f(\chi+\chi_b)}}p(\phi_1,\phi_2)
\sum_{c\in\supp(\varphi-\phi')}\sum_{\substack{d\in\supp\phi'\\ d'\in\supp\chi}}(\chi(b)-\chi_{d'}(b)+1)
\sum_{\substack{\phi_1'\in\f(\phi'-\phi_1)\\ \phi_2'\in\f(\chi+\chi_b-\phi_2)\\ |\phi_1'|=|\phi_2'|}}\\
&\times&\frac{\m\left(\phi_1'\right)}{|\phi_1'|}\m\left(\phi_2'\right)\phi_1'(d)\phi_2'(d')
(x^+\otimes c\pi(\phi_1')\pi(\phi_2'))\\
&\times&D^+(\phi'-\phi_1-\phi_1',\chi+\chi_b-\phi_2-\phi_2',\varphi-\phi'-\chi_c)\\
&+&\sum_{\phi'\in\f(\varphi)}\sum_{\substack{\phi_1\in\f(\phi')\\ \phi_2\in\f(\chi+\chi_b)}}
p(\phi_1,\phi_2)(|\varphi|-|\chi|)(\chi(b)+1)D^+(\phi'-\phi_1,\chi+\chi_b-\phi_2,\varphi-\phi')
\end{eqnarray*}
So it suffices to show that for $\phi',\phi_1,\phi_2\in \f$
with $\phi_1\leq\phi'$, $\phi_2\leq\chi+\chi_b$, and
$|\phi_1|=|\phi_2|$
\begin{eqnarray*}
&&\hskip-.6in\phi_2(b)D^+(\phi'-\phi_1,\chi+\chi_b-\phi_2,\varphi-\phi')\\
&=&-(\chi(b)+1)\sum_{c\in\supp(\varphi-\phi')}
\sum_{\substack{\phi_1'\in\f(\phi'-\phi_1)\\ \phi_2'\in\f(\chi+\chi_b-\phi_2)\\ |\phi_1'|=|\phi_2'|}}(|\phi_1'|+1)
\m\left(\phi_1'\right)\m\left(\phi_2'\right)(x^+\otimes c\pi(\phi_1')\pi(\phi_2'))\\
&\times&D^+(\phi'-\phi_1'-\phi_1,\chi+\chi_b-\phi_2'-\phi_2,\varphi-\phi'-\chi_c)\\
&+&\sum_{c\in\supp(\varphi-\phi')}\sum_{\substack{d\in\supp\phi'\\ d'\in\supp\chi}}(\chi(b)-\chi_{d'}(b)+1)
\sum_{\substack{\phi_1'\in\f(\phi'-\phi_1)\\ \phi_2'\in\f(\chi+\chi_b-\phi_2)\\ |\phi_1'|=|\phi_2'|}}
\frac{\m\left(\phi_1'\right)}{|\phi_1'|}\m\left(\phi_2'\right)\phi_1'(d)\phi_2'(d')\\
&\times&(x^+\otimes c\pi(\phi_1')\pi(\phi_2'))
D^+(\phi'-\phi_1-\phi_1',\chi+\chi_b-\phi_2-\phi_2',\varphi-\phi'-\chi_c)\\
&+&(|\varphi|-|\chi|)(\chi(b)+1)D^+(\phi'-\phi_1,\chi+\chi_b-\phi_2,\varphi-\phi')
\end{eqnarray*}
\begin{eqnarray*}
&=&-\sum_{c\in\supp(\varphi-\phi')}\sum_{\substack{\phi_1'\in\f(\phi'-\phi_1)\\ \phi_2'\in\f(\chi+\chi_b-\phi_2)}}
\phi_2'(b)D^+(\phi_1',\phi_2',\chi_c)\\
&\times&D^+(\phi'-\phi_1-\phi_1',\chi+\chi_b-\phi_2-\phi_2',\varphi-\phi'-\chi_c)\\
&+&(\chi(b)+1)D^+(\phi'-\phi_1,\chi+\chi_b-\phi_2,\varphi-\phi')
\end{eqnarray*}
So the proof is reduced to showing that
\begin{eqnarray*}
&&\hskip-.7in(\chi(b)+1-\phi_2(b))D^+(\phi'-\phi_1,\chi+\chi_b-\phi_2,\varphi-\phi')\\
&=&\sum_{c\in\supp(\varphi-\phi')}\sum_{\substack{\phi_1'\in\f(\phi'-\phi_1)\\ \phi_2'\in\f(\chi+\chi_b-\phi_2)}}
\phi_2'(b)D^+(\phi_1',\phi_2',\chi_c)\\
&\times&D^+(\phi'-\phi_1-\phi_1',\chi+\chi_b-\phi_2-\phi_2',\varphi-\phi'-\chi_c)
\end{eqnarray*}
but this is true by Proposition \ref{idD}$(i)$. We have
completed the proofs of Lemma \ref{basic} and Theorem
\ref{thm}.
\end{proof}



\end{document}